\documentclass[letterpage, 11pt, notitlepage]{article}

\usepackage[margin=1.0in]{geometry}

\usepackage{amsthm,amsmath, amsfonts, amssymb, graphicx,epstopdf,color,soul,mathtools,enumitem,tabularx,tabulary,booktabs}
\usepackage{mathrsfs, mathtools,bbm, dsfont}

\usepackage[bookmarks=true,pageanchor,colorlinks,linkcolor=red,anchorcolor=blue, citecolor=blue,urlcolor=blue,hyperfootnotes=false]{hyperref}

\title{\LARGE Convergence Rates of Two-Time-Scale Gradient Descent-Ascent Dynamics for Solving Nonconvex Min-Max Problems}
\author{
Thinh T. Doan\thanks{Thinh T. Doan is with the Bradley Department of Electrical and Computer Engineering, Virginia Tech, USA. Email: {\tt\small thinhdoan@vt.edu}}
}

\newcommand{\Rset}{\mathbb{R}}

\newcommand{\Ocal}{{\cal O}}
\newcommand{\Pcal}{{\cal P}}

\newcommand{\Ycal}{{\cal Y}}

\newcommand{\Ibf}{{\bf I}}

\newcommand{\xdot}{\dot{x}}
\newcommand{\ydot}{\dot{y}}

\newtheorem{lem}{Lemma}
\newtheorem{thm}{Theorem}

\newtheorem{definition}{Definition}

\newtheorem{assump}{Assumption}

\date{}
\begin{document}

\maketitle

\begin{abstract}%
There are much recent interests in solving noncovnex min-max optimization problems due to its broad applications in many areas including machine learning, networked resource allocations, and distributed optimization. Perhaps, the most popular first-order method in solving min-max optimization is the so-called simultaneous (or single-loop) gradient descent-ascent algorithm due to its simplicity in implementation. However, theoretical guarantees on the convergence of this algorithm is very sparse since it can diverge even in a simple bilinear problem. 

In this paper, our focus is to characterize the finite-time performance (or convergence rates) of the continuous-time variant of simultaneous gradient descent-ascent algorithm. In particular, we derive the rates of convergence of this method under a number of different conditions on the underlying objective function, namely, two-sided Polyak-\L ojasiewicz (P\L), one-sided P\L, nonconvex-strongly concave, and strongly convex-nonconcave conditions. Our convergence results improve the ones in prior works under the same conditions of objective functions. The key idea in our analysis is to use the classic singular perturbation theory and coupling Lyapunov functions to address the time-scale difference and interactions between the gradient descent and ascent dynamics. Our results on the behavior of continuous-time algorithm may be used to enhance the convergence properties of its discrete-time counterpart.   
\end{abstract}



\section{Introduction}
In this paper, we consider the following min-max optimization problems 
\begin{align}
    \min_{x\in\Rset^{m}}\max_{y\in\Rset^{n}}f(x,y),\label{prob:minmax} 
\end{align}
where $f:\Rset^{m}\times \Rset^{n}\rightarrow\Rset$ is a nonconvex function w.r.t $x$ for a fixed $y$ and (possibly) nonconcave w.r.t $y$ for a fixed $x$. The min-max problem has received much interests for years due to its broad applications in different areas including control, machine learning, and economics. In particular, many problems in these areas can be formulated as problem \eqref{prob:minmax}, for example, game theory \cite{Basar_book_gametheory_1998,Shapley1953StochasticG}, stochastic control and reinforcement learning \cite{altman1999constrained,achiam17a}, training generative adversarial networks (GANs) \cite{Goodfellow_GAN_2020,Mescheder2017}, adversarial and robust machine learning \cite{KurakinGB17,Qian_Zhu_Tang_Jin_Sun_Li_2019}, resource allocation over networks \cite{6450111}, and distributed optimization \cite{Lan2020_DecentralizedOpt,9085431}; to name just a few.      

In the existing literature, there are two types of iterative first-order methods for solving problem \eqref{prob:minmax}, namely, nested-loop algorithms and single-loop algorithms. Nested-loop algorithms implement multiple inner steps in each iteration to solve the maximization problem either exactly or approximately. However, this approach is not applicable to the setting when $f(x,y)$ is nonconcave in $y$, since the maximization problem is NP-hard. Only finding a stationary point of the maximization problem is likely to affect the quality of solving the minimization problem.          

On the other hand, single-loop algorithm simultaneously updates the iterates $x$ and $y$ by using the vanilla gradient descent and ascent steps at different time scales, respectively. As a result, this algorithm is applicable to more general settings and more practical due to its simplicity in implementation. However, single-loop algorithms may not converge in many settings, for example, they fail to converge even in a simple bilinear zero-sum game \cite{balduzzi18a}. Indeed, theoretical guarantees of these methods are very sparse. 

Our focus in this paper is to study the continuous-time variant of the single-loop gradient descent-ascent method for solving problem \eqref{prob:minmax}. Considering the continuous-time variant will help us to have a better understanding about the behavior of this method through studying the convergence of the corresponding differential equations using Lyapunov theory. Such an understanding can then be used to enhance the analysis of the discrete-time algorithms, as recently observed in the single objective optimization counterpart \cite{Krichene_NIPS2015,6426639,Weijie_NeurIPS2014,Diakonikolas_2019}. Our main contributions are summarized below.\vspace{0.1cm}

\noindent\textbf{Main Contributions.} The focus of this paper is to study the performance of the continuous-time gradient descent-ascent dynamics in solving nonconvex min-max optimization problems. In particular, we derive the rates of convergence of this method under a number of different conditions on the underlying objective function, namely, two-sided Polyak-\L ojasiewicz (P\L), one-sided P\L, nonconvex-strongly concave, and strongly convex-nonconcave conditions. These rates are summarized in Table \ref{table:comparison} and presented in detail in Section \ref{sec:results}, where we show that our results improve the ones in prior works under the same conditions of objective functions. The key idea in our analysis is to use the classic singular perturbation theory and coupling Lyapunov function of the fast and slow dynamics to address the time-scale difference and interactions between the gradient descent and ascent dynamics. Proper choices of step sizes allows us to derive improved convergence properties of the two-time-scale gradient descent-ascent dynamics.    

\subsection{Related Works}\label{subsec:related_works}

\paragraph{Convex-Concave Settings.} 
Given the broad applications of problem \eqref{prob:minmax}, there are a large number of works to study algorithms and their convergence in solving this problem, especially in the context of convex-concave settings. Some examples include prox-method and its variant \cite{Nemirovski_SIAM2004,Malitsky_SIAM2015,Wang_NEURIPS2020, Cherukuri_2017,TANG2020104754}, extragradient and optimistic gradient methods \cite{Korpelevich1976,Mokhtari_SIAM2020,Monteiro_SIAM2010,Golowich20a,yoon21d,Dang2015OnTC}, and recently Hamiltonian gradient descent methods \cite{Mescheder2017,balduzzi18a,abernethy21a}. Some algorithms in these settings have convergence rates matched with the lower bound complexity; see the recent work \cite{yoon21d} for a detailed discussion.

\paragraph{Nonconvex-Concave Settings.} Unlike the convex-concave settings, algorithmic development and theoretical understanding in the general nonconvex settings are very limited. Indeed, finding the global optimality of nonconvex-nonconcave problem is NP-hard, or at least as hard as solving a single nonconvex objective problem. As a result, the existing literature often aims to find a stationary point of $f$ when the max problem is concave. For example, multiple-loop algorithms have been studied in \cite{Thekumparampil_NEURIPS2019,Kong_SIAM2021,Rafique_OMS2021,lin20b,Nouiehed_NEURIPS2019}. Our work in this paper is closely related to the recent literature on studying single-loop algorithm \cite{lin20a,9070155,Yang_NEURIPS2020,Xu2020AUS,Zhang_NEURIPS2020}. While these works study discrete-time algorithms, we consider continuous-time counterpart. We will show that for some settings, our approach improves the existing convergence results.

\paragraph{Other Settings.} We also want to mention some related literature in game theory \cite{loizou20a,Zhang_NEURIPS2019,Cen2021FastPE,perolat18a,Daskalakis_Independent_NEURIPS2020}, two-time-scale stochastic approximation \cite{borkar2008,KondaT2004, Dalal_Szorenyi_Thoppe_2020, DoanR2019,GuptaSY2019_twoscale,Doan_two_time_SA2019,Kaledin_two_time_SA2020,MokkademP2006,doan2021finite,doan2020nonlinear}, reinforcement learning \cite{borkar2005actor,bhatnagar2012online,Paternain_NEURIPS2019,ding2020natural,Zeng_CMDP_2021}, two-time-scale optimization \cite{WangFL2017,Zeng_TwoTimeScaleOpt_2021}, and decentralized optimization \cite{6461383,8786146,Doan_random_quantization2020,Doan_delays_2017,Dutta_2021,Dutta_Decentralized_Opt2021,Vancocelos_CDC2021}. These works study different variants of two-time-scale methods mostly for solving a single optimization problem, and often aim to find global optimality (or fixed points) using different structure of the underlying problems (e.g., Markov structure in stochastic games and reinforcement learning or strong monotonicity in stochastic approximation). As a result, their techniques may not be applicable to the context of problem \eqref{prob:minmax} considered in the current paper.\vspace{0.1cm}

\noindent\textbf{Notation.} Given any vector $x$ we use $\|x\|$ to denote its $2$-norm. We denote by $\nabla_{x} f(x,y)$ and $\nabla_{y} f(x,y)$ the partial gradients of $f$ with respect to $x$ and $y$, respectively.


\section{Two-Time-Scale Gradient Descent-Ascent Dynamics}
For solving problem \eqref{prob:minmax}, we are interested in studying two-time-scale gradient descent-ascent dynamics (GDAD), where we implement simultaneously the following two differential equations
\begin{align}
\begin{aligned}
\xdot(t) &= \frac{d}{dt}x(t) = -\alpha\nabla_{x} f(x(t),y(t)),\\   
\ydot(t) &= \frac{d}{dt}y(t) = \beta\nabla_{y} f(x(t),y(t)),   
\end{aligned}    \label{alg:GDAD}
\end{align}
Here, $\alpha,\beta$ are two step sizes, whose values will be specified later. In the convex-concave setting, one can choose $\alpha=\beta$. However, as observed in  \cite{Heusel_NIPS2017}, choosing different step sizes achieves a better convergence in the context of nonconvex problem. Indeed, we will choose $\alpha \ll \beta$ since in our settings studied in the following sections, the maximization problem is often easier to solve than the minimization problem. In this case, the dynamic of $y(t)$ is implemented at a faster time scale (using larger step sizes) than $x(t)$ (using smaller step sizes). The time-scale difference is loosely defined as the ratio $\alpha/\beta \ll 1$. Thus, one has to design these two step sizes properly so that the method converges as fast as possible.\vspace{0.1cm}    

\noindent\textbf{Technical Approach.} The convergence analysis of \eqref{alg:GDAD} studied in this paper is mainly motivated by the classic singular perturbation theory \cite{Kokotovic_SP1999}. The main idea of our approach can be explained as follows. Since $y$ is implemented at a faster time scale than $x$, one can consider $x(t) = x$ being fixed in $\ydot$ and separately study the stability of the system $\ydot$ using Lyapunov theory. Let $V_{2}$ be the Lyapunov function corresponding to $\ydot$. When $\ydot$ converges to an equilibrium $y$ (e.g., $\nabla _{y}f(x,y) = 0$), one can fix $y(t) = y$ and study the stability of $\xdot$. Let $V_{1}$ be the  corresponding Lyapunov function of $\xdot$. We note that $V_{1}$ and $V_{2}$ both depend on $x$ and $y$, as a result, their time derivatives are coupled through the dynamics in \eqref{alg:GDAD}. Addressing this coupling and the time-scale difference between the two dynamics is the key idea in our approach. To do that, we will consider the following Lyapunov function
\begin{align}
V(x,y) = V_{1}(x,y) + \frac{\gamma\alpha}{\beta} V_{2}(x,y),\label{def:Lyapunov_V}  
\end{align}
where $\alpha/\beta$ represents the time-scale difference, while the constant $\gamma$ will be properly chosen to eliminate the impact of $x$ on the convergence of $y$ and vice versa. Proper choices of these constants will also help us to derive the convergence rates of \eqref{alg:GDAD}. Similar approach has been used in different settings of two-time-scale methods, see for example \cite{doan2020nonlinear,Dutta_Decentralized_Opt2021}.

We conclude this section by introducing two assumptions for our analysis studied later.


\begin{assump}\label{assump:Lipschitz_cont}
The function $f(\cdot,\cdot)$ has Lipschitz continuous gradients for each variable, i.e., there exist positive constants $L_{x}$, $L_{y}$, and $L_{xy}$ such that for all $x_{1},x_{2}\in\Rset^{m},y_{1},y_{2}\in\Rset^{n}$ we have
\begin{align}
    \begin{aligned}
\|\nabla_{x} f(x_{1},y_{1}) - \nabla_{x}f(x_{2},y_{2})\| &\leq L_{x}\|x_{1}-x_{2}\| + L_{xy}\|y_{1}-y_{2}\|,\\ 
\|\nabla_{y} f(x_{1},y_{1}) - \nabla_{y}f(x_{2},y_{2})\| &\leq L_{xy}\|x_{1}-x_{2}\| + L_{y}\|y_{1}-y_{2}\|.    \end{aligned}
\end{align}\label{assump:Lipschitz_cont:ineq}
\end{assump}
\begin{assump}\label{assump:max}
Given any $x$ the problem $\max_{y}f(x,y)$ has a nonempty solution set $\Ycal^{\star}(x)$, i.e., there exists $y^{\star}(x)\in\Ycal^{\star}(x)$ such that
\begin{align*}
y^{\star}(x) = \arg\max_{y\in\Rset^{n}}f(x,y),\quad \text{where } f(x,y^{\star}(x)) \text{ is finite}. 
\end{align*}
\end{assump}\vspace{-0.8cm}

\begin{table}[h]
\centering
\caption{Convergence rates of GDAD for solving \eqref{prob:minmax} given some accuracy $\epsilon > 0$. The abbreviations NCvex, NCave, SCvex, SCave, and P\L\ stand for nonconvex, nonconcave, strongly convex, strongly concave, and Polyak-\L ojasiewicz condtions, respectively. Condition number $\kappa$ is defined in \eqref{def:condition_number}, and $R$ is the size of compact set used in \cite{Nouiehed_NEURIPS2019}.}\vspace{0.3cm}
\begin{small}
\begin{sc}
\begin{tabular}{l|c|c}
\toprule
\textbf{Objectives} & \text{Prior Works} & This Paper \\
\midrule
\textbf{P\L \& P\L} & $\Ocal\left(\kappa^{3}\log(\frac{1}{\epsilon})\right)$ {\tiny\cite{Yang_NEURIPS2020}} & $\Ocal\left(\kappa^{2}\log(\frac{1}{\epsilon}\right)$\\
\hline
\textbf{NCvex \& P\L} & $\Ocal\left(R^2L_{xy}\log(\frac{1}{\epsilon})\epsilon^{-2}\right)$ {\tiny  \cite{Nouiehed_NEURIPS2019}}  & $\Ocal\left(L_{xy}^2\epsilon^{-2}\right)$\\
\hline
\textbf{NCvex \& SCave}  & $\Ocal\left(L_{xy}^2\epsilon^{-2}\right)$ {\tiny \cite{Xu2020AUS}} & $\Ocal\left(L_{xy}^2\epsilon^{-2}\right)$\\
\hline
\textbf{SCvex \& NCave}  & $\Ocal\left(L_{xy}^2\epsilon^{-2}\right)$ {\tiny \cite{Xu2020AUS}} & $\Ocal\left(L_{xy}^2\epsilon^{-2}\right)$\\
\bottomrule
\end{tabular}\vspace{-0.2cm}
\end{sc}
\end{small}
\label{table:comparison}
\end{table}
\section{Main Results}\label{sec:results}
In this section, we present the main results of this paper, where we derive the convergence rates of GDAD under different conditions on the objective function $f(x,y)$. Our results are summarized in Table \ref{table:comparison}. First, our approach improves the analysis in \cite{Yang_NEURIPS2020}, where we show in Section \ref{subsec:2-sided-PL} that for two-sided P\L\ functions the convergence of GDAD only scales with $\kappa^2$ instead of $\kappa^3$ studied in \cite{Yang_NEURIPS2020}. Our result addresses the conjecture raised in \cite{Yang_NEURIPS2020}, where the authors state that such an improvement may not be possible. Second, our analysis achieves a better result than the one in \cite{Nouiehed_NEURIPS2019} for the case of one-sided PL function by a factor of $\log(1/\epsilon)$. We note that a nested-loop is studied in \cite{Nouiehed_NEURIPS2019} while GDAD is a single-loop method. Finally, our result is the same as the one in \cite{Xu2020AUS} when $f(x,y)$ is either strongly concave in $y$ for fixed $x$. In Section \ref{subsec:stronglyconvex}, we will show that this observation also holds when $f(x,y)$ is either strongly convex in $x$ and nonconcave in $y$. Note that as compared to the analysis in \cite{Xu2020AUS}, we use a simpler analysis and simpler choice of step sizes to achieve these results.


\subsection{Two-Sided Polyak--\L ojasiewicz Conditions}\label{subsec:2-sided-PL}
We first study the convergence rates of GDAD when $f$ satisfies a two-sided Polyak--\L ojasiewicz (P\L) condition, which is considered in \cite{Yang_NEURIPS2020} and stated here for convenience.  

\begin{definition}[Two-Sided P\L\ Conditions]\label{def:two-sided-PL}
A continuously differentiable function $f:\Rset^{m}\times\Rset^{n}\rightarrow\Rset$ is called to satisfy two-sided P\L\ conditions if there exist two positive constants $\mu_{x}$ and $\mu_{y}$ such that $\mu_{x},\mu_{y} \leq \min\{L_{x},L_{y},L_{xy}\}$ the following conditions hold for all $(x,y)\in\Rset^{m}\times\Rset^{n}$:
\begin{align}
\begin{aligned}
2\mu_{x}[f(x,y) - \min_{x}f(x,y)] &\leq \|\nabla_{x}f(x,y)\|^2,\\
2\mu_{y}[\max_{y}f(x,y) -  f(x,y)] &\leq \|\nabla_{y}f(x,y)\|^2.
\end{aligned}\label{def:two-sided-PL:ineq}
\end{align}
\end{definition}
The two-sided P\L\ condition, which we will assume to hold in this subsection, is a generalized variant of the popular P\L\ condition, proposed by \cite{Polyak_PL1963} as a sufficient condition to guarantee that the classic gradient descent method converges exponentially to the optimal value of an unconstrained minimization problem. As shown in \cite{Karimi_PL_2016}, the P\L\ condition also implies the quadratic growth condition, i.e., given any $x$ we have
\begin{align}
\max_{z\in\Rset^{m}}f(x,z) - f(x,y) \geq \frac{\mu_{y}}{2}\|\Pcal_{\Ycal^{\star}(x)}[y]-y\|^2,\quad \forall y\in\Rset^{m},\label{eq:QG}   
\end{align}
where we assume that $\Ycal^{\star}(x)$ is a nonempty solution set of $\max_{y}f(x,y)$ and $\Pcal_{\Ycal^{\star}(x)}[y]$ is the projection of $y$ to this set. More discussions on P\L\ condition can be found in \cite{Karimi_PL_2016}, while some examples of functions satisfying the two-sided P\L\ condition are given in \cite{Yang_NEURIPS2020}.

Our focus in this section is to show that GDAD converges exponentially to the global min-max solution $(x^{\star},y^{\star})$ of $f$ under the two-sided P\L\ condition. To do that, we consider the following assumption and lemmas, which are useful for our analysis considered later. We first consider an assumption on the existence of $(x^{\star},y^{\star})$, a global min-max solution of $f$.

\begin{assump}\label{assump:global_minmax}
There exists a global min-max solution $(x^{\star},y^{\star})$ of $f$, i.e., 
\begin{align*}
x^{\star} = \arg\min_{x\in\Rset^{m}}f(x,y^{\star})\quad \text{and}\quad y^{\star} = \arg\max_{y\in\Rset^{n}}f(x^{\star},y).    
\end{align*}
\end{assump}
Next, we consider the following lemma about the Lipschitz continuity of the gradient of $f(x,y^{\star}(x))$, which is a variant of the well-known Danskin lemma \cite{Bertsekas1999}[Proposition B.25] and studied in \cite{Nouiehed_NEURIPS2019}[Lemma A.5].  
\begin{lem}\label{lem:Nouiehed_NEURIPS2019}
Suppose that Assumptions \ref{assump:Lipschitz_cont}-- \ref{assump:global_minmax} hold. Then, the function $\max_{y}f(x,y)$ is differentiable and its gradient $\nabla_{x}f(x,y^{\star}(x))$ is Lipschitz continuous with a constant $L_{x} + \frac{L_{xy}}{\mu_{y}}$.
\end{lem}
Finally, for our analysis we consider the following two Lyapunov functions
\begin{align}
V_{1}(x) &=  \max_{y\in\Rset^{n}}f(x,y) - \min_{x\in\Rset^{m}}\max_{y\in\Rset^{n}}f(x,y),\label{two-sided-PLLyapunov_V1}\\  
V_{2}(x,y) &= \max_{y}f(x,y) - f(x,y),\label{two-sided-PLLyapunov_V2}
\end{align}
where it is obvious to see that $V_{1}$ and $V_{2}$ are nonnegative. The time derivatives of $V_{1}$ and $V_{2}$ over the trajectories $\xdot$ and $\ydot$ are given in the following lemma, whose proof can be found in Section \ref{analysis:proof:two-sided-PL:lem:V12_dot}. 

\begin{lem}\label{two-sided-PL:lem:V12_dot}
Suppose that Assumptions \ref{assump:Lipschitz_cont}-- \ref{assump:global_minmax} hold. Then we have
\begin{align}
\dot{V}_{1}(x) &\leq -\frac{\alpha}{2}\|\nabla_{x} f(x,y^{\star}(x))\|^2  + \frac{L_{xy}^2\alpha}{\mu_{y}}V_{2}(x,y).   \label{two-sided-PLlem:V12_dot:ineq1}\\
\dot{V}_{2}(x,y) &\leq - \beta\|\nabla_{y}f(x,y)\|^2 + \frac{3\alpha}{2} \|\nabla_{x} f(x,y^{\star}(x))\|^2+ \frac{5L_{xy}^2\alpha}{\mu_{y}}V_{2}(x,y)    \label{two-sided-PLlem:V12_dot:ineq2}.
\end{align}
\end{lem}

As mentioned, the dynamics of $\xdot$ and $\ydot$ are implemented at different time scales, where this difference is often loosely defined as the ratio $\beta/\alpha > 1$. To capture such time-scale difference in our analysis, we will utilize the coupling Lyapunov function defined in \eqref{def:Lyapunov_V}. We denote by $\mu = \min\{\mu_{x},\mu_{y}\}$ and the condition number 
\begin{align}
    \kappa = \frac{L_{xy}}{\mu} \geq 1.\label{def:condition_number}
\end{align}
representing the condition number of $f(x,y)$. The convergence rate of GDAD under the two-sided P\L\ condition is formally stated in the following theorem.

\begin{thm}\label{thm:PL}
Suppose that Assumptions \ref{assump:Lipschitz_cont}-- \ref{assump:global_minmax} hold. Let $\gamma,\alpha,\beta$ be chosen as 
\begin{align}
\gamma = \frac{L_{xy}^2}{\mu_{y}^2},\quad \alpha = \frac{\mu^2}{10\mu_{x}L_{xy}^2},\quad \beta = \frac{\mu^2}{\mu_{x}\mu_{y}^2}\cdot     \label{two-sided-PL:thm:stepsize}
\end{align}
Then we have for all $t\geq0$
\begin{align}
V(x(t),y(t)) \leq  e^{-\frac{t}{20\kappa^2}}V(x(0),y(0)).    \label{two-sided-PLthm:rate}
\end{align}
\end{thm}

\begin{proof}
By \eqref{def:two-sided-PL:ineq} we have
\begin{align*}
\|\nabla_{y}f(x,y)\|^2 \geq 2\mu_{y}[\max_{y}f(x,y) - f(x,y)] = 2\mu_{y}V_{2}(x,y).     
\end{align*}
Thus, by using \eqref{two-sided-PLlem:V12_dot:ineq1}, \eqref{two-sided-PLlem:V12_dot:ineq2}, \eqref{def:Lyapunov_V}, and the preceding relation we have
\begin{align}
\dot{V}(x(t),y(t)) &=  \dot{V}_{1}(x(t)) + \frac{\gamma\alpha}{\beta} \dot{V}_{2}(x(t),y(t))\notag\\ 
&\leq -\frac{\alpha}{2}\|\nabla_{x} f(x(t),y^{\star}(x(t)))\|^2  + \frac{L_{xy}^2\alpha}{\mu_{y}}V_{2}(x(t),y(t))\notag\\
&\quad - 2\mu_{y}\gamma\alpha V_{2}(x(t),y(t)) + \frac{3\gamma\alpha^2}{2\beta} \|\nabla_{x} f(x(t),y^{\star}(x(t)))\|^2+ \frac{5L_{xy}^2\gamma\alpha^2}{\mu_{y}\beta}V_{2}(x(t),y(t))\notag\\
&= -\frac{\alpha}{4}\|\nabla_{x} f(x(t),y^{\star}(x(t)))\|^2 - \frac{\mu_{y}\gamma\alpha}{2} V_{2}(x(t),y(t))\notag\\ 
&\quad - \left(\frac{1}{2} - \frac{3\gamma\alpha}{\beta}\right)\frac{\alpha}{2}\|\nabla_{x} f(x(t),y^{\star}(x(t)))\|^2\notag\\ 
&\quad - \left(\frac{3\mu_{y}\gamma}{2} - \frac{L_{xy}^2}{\mu_{y}} - \frac{5L_{xy}^2\gamma\alpha}{\mu_{y}\beta}\right)\alpha V_{2}(x(t),y(t)).\label{two-sided-PLthm:Eq1}
\end{align}
Using \eqref{two-sided-PL:thm:stepsize} we have 
\begin{align*}
&\frac{1}{2} - \frac{3\gamma\alpha}{\beta} = \frac{1}{2} - \frac{3L_{xy}^2}{\mu_{y}^2}\frac{\mu^2}{10\mu_{x}L^2_{xy}}\frac{\mu_{x}\mu_{y}^2}{\mu^2} = \frac{1}{5},\\
&\frac{3\mu_{y}\gamma}{2} - \frac{L_{xy}^2}{\mu_{y}} - \frac{5L_{xy}^2\gamma\alpha}{\mu_{y}\beta} = \frac{L_{xy}^2}{2\mu_{y}} - \frac{L_{xy}^2}{2\mu_{y}} = 0,
\end{align*}
which when substituting into \eqref{two-sided-PLthm:Eq1} and using \eqref{def:two-sided-PL:ineq} and $y^{\star}(x) = \arg\max_{y}f(x,y)$ we obtain
\begin{align*}
\dot{V}(x(t),y(t)) &\leq -\frac{\alpha}{4}\|\nabla_{x} f(x(t),y^{\star}(x(t)))\|^2 - \frac{\mu_{y}\gamma\alpha}{2} V_{2}(x(t),y(t))\\
&\leq -\frac{\mu_{x}\alpha}{2}\Big[\max_{y}f(x(t),y) - \min_{x}f(x,y^{\star}(x(t)))\Big] - \frac{\mu_{y}\gamma\alpha}{2} V_{2}(x,y)\\
&\leq -\frac{\mu_{x}\alpha}{2}(\max_{y}f(x(t),y) - \min_{x}\max_{y}f(x,y)) - \frac{\mu_{y}\gamma\alpha}{2} V_{2}(x,y)\\
&= -\frac{\mu_{x}\alpha}{2}V_{1}(x(t)) - \frac{\mu_{y}\beta}{2}\frac{\gamma\alpha}{\beta}V_{2}(x,y)\\
&\leq -\frac{\mu_{x}\alpha}{2}(V_{1}(x(t)) + \frac{\gamma\alpha}{\beta} V_{2}(x(t),y(t))) = -\frac{\mu_{x}\alpha}{2}V(x(t),y(t)),
\end{align*}
where the last inequality is due to 
\begin{align*}
\mu_{x}\alpha = \frac{\mu^2}{10L_{xy}^2} \leq \mu_{y}\beta = \frac{\mu^2}{\mu_{x}\mu_{y}}\cdot    
\end{align*}
Taking the integral on both sides of the equation above immediately gives \eqref{two-sided-PLthm:rate}, i.e., 
\begin{align*}
V(x(t),y(t)) \leq e^{-\frac{\mu_{x}\alpha t}{2}}V(x(0),y(0)) =  e^{-\frac{t}{20\kappa^2}}V(x(0),y(0)).   
\end{align*}
\end{proof}


\subsection{Nonconvex--Polyak-\L ojasiewicz Conditions}\label{subsec:One-Sided-PL}
In this subsection, we consider an extension of the result studied in the previous section, where we assume that the objective function $f(x,\cdot)$ satisfies the Polyak-\L ojasiewicz condition given any $x$ and $f(\cdot,y)$ is nonconvex given any $y$.

\begin{assump}[One-Sided P\L\ Conditions]\label{assump:one-sided-PL}
We assume that $f:\Rset^{m}\times\Rset^{n}\rightarrow\Rset$ is nonconvex in $x$ for any fixed $y$ and satisfies the P\L\ condition in $y$ for any fixed $x$, that is, there exists a positive constants $\mu_{y}$ such that the following condition hold for any $x\in\Rset^{m}$:
\begin{align}
\begin{aligned}
2\mu_{y}[\max_{y}f(x,y) -  f(x,y)] \leq \|\nabla_{y}f(x,y)\|^2.
\end{aligned}\label{assump:one-sided-PL:ineq}
\end{align}
\end{assump}
Since $f$ satisfies only one-sided P\L\ condition, we are giving up the hope to find a global optimal solution of \eqref{prob:minmax}, as studied in Theorem \ref{thm:PL}. In stead, we will show that GDAD will return a stationary point of $f$, as studied in \cite{Nouiehed_NEURIPS2019}. Note that under Assumption \ref{assump:max} the result in Lemma \ref{lem:Nouiehed_NEURIPS2019} still holds since the work in \cite{Nouiehed_NEURIPS2019} only assumes one-sided P\L\ condition. In addition, since we relax the two-sided P\L\ condition,  we introduce the following two Lyapunov functions for our analysis studied later.
\begin{align}
V_{1}(x,y) &=  f(x,y) - \min_{x\in\Rset^{m}}\min_{y\in\Rset^{n}}f(x,y),\label{One_sided_PL:Lyapunov_V1}\\  
V_{2}(x,y) &= \max_{y}f(x,y) - f(x,y),\label{One_sided_PL:Lyapunov_V2}
\end{align}
where it is obvious to see that $V_{1}$ and $V_{2}$ are nonnegative. The time derivatives of $V_{1}$ and $V_{2}$ over the trajectories $\xdot$ and $\ydot$ are given in the following lemma, whose proof is presented in Section \ref{analysis:proof:one-sided-PL:lem:V12_dot}. 

\begin{lem}\label{One_sided_PL:lem:V12_dot}
Suppose that Assumptions \ref{assump:Lipschitz_cont}, \ref{assump:max}, and \ref{assump:one-sided-PL} hold. Then we have
\begin{align}
\dot{V}_{1}(x(t),y(t)) &= -\alpha\|\nabla_{x}f(x(t),y(t))\|^2 + \beta\|\nabla_{y}f(x(t),y(t))\|^2.   \label{One_sided_PL:lem:V12_dot:ineq1}\\
\dot{V}_{2}(x(t),y(t)) &\leq - \beta\|\nabla_{y}f(x(t),y(t))\|^2 + \frac{\alpha}{2}\|\nabla_{x} f(x(t),y(t))\|^2 + \frac{L_{xy}^2\alpha}{\mu_{y}}V_{2}(x(t),y(t))   \label{One_sided_PL:lem:V12_dot:ineq2}.
\end{align}
\end{lem}
Similar to the previous subsection, we utilize the following coupling Lyapunov function 
\begin{align}
V(x,y) = V_{1}(x,y) + \frac{\gamma\alpha}{\beta} V_{2}(x,y),\label{One_sided_PL:Lyapunov_V}  
\end{align}
for some constant $\gamma$, which will be defined below. The convergence rate of GDAD under the nonconvex-P\L\ condition is formally stated in the following theorem.

\begin{thm}\label{one_sided_PL:thm}
Suppose that Assumptions \ref{assump:Lipschitz_cont},  \ref{assump:max},  and \ref{assump:one-sided-PL} hold. Let $\gamma,\alpha,\beta$ be chosen as 
\begin{align}
\gamma = \frac{32L_{xy}^2}{\mu_{y}^2},\quad \alpha = \frac{1}{8L_{xy}^2},\quad \beta = \frac{1}{\mu_{y}^2}\cdot     \label{one_sided_PL:thm:stepsize}
\end{align}
Then we have for all $T\geq0$
\begin{align}
\min_{0\leq t\leq T}\left\|\begin{array}{cc}
 \nabla_{x}f(x(t),y(t)) \\
 \nabla_{y}f(x(t),y(t)) 
\end{array}\right\|
&\leq \frac{4L_{xy}\sqrt{V_{1}(x(0),y(0)) +4V_{2}(x(0),y(0))}}{\sqrt{T}}\cdot   \label{one_sided_PL:thm:rate}
\end{align}
\end{thm}

\begin{proof}
By using \eqref{One_sided_PL:lem:V12_dot:ineq1}, \eqref{One_sided_PL:lem:V12_dot:ineq2}, and \eqref{One_sided_PL:Lyapunov_V} we have
\begin{align}
\dot{V}(x(t),y(t)) &=  \dot{V}_{1}(x(t)) + \frac{\gamma\alpha}{\beta} \dot{V}_{2}(x(t),y(t))\notag\\ 
&\leq -\alpha\|\nabla_{x}f(x(t),y(t))\|^2 + \beta\|\nabla_{y}f(x(t),y(t))\|^2\notag\\
&\quad - \gamma\alpha\|\nabla_{y}f(x(t),y(t))\|^2 + \frac{\gamma\alpha^2}{8\beta}\|\nabla_{x} f(x(t),y(t))\|^2 + \frac{4L_{xy}^2\gamma\alpha^2}{\mu_{y}\beta}V_{2}(x,y)\notag\\
&= -\frac{\alpha}{2}\|\nabla_{x}f(x(t),y(t))\|^2 - \frac{\gamma\alpha}{2}\|\nabla_{y}f(x(t),y(t))\|^2\notag\\
&\quad - \frac{\gamma\alpha}{4}\|\nabla_{y}f(x(t),y(t))\|^2 + \beta\|\nabla_{y}f(x(t),y(t))\|^2\notag\\ 
&\quad - \frac{\gamma\alpha}{4}\|\nabla_{y}f(x(t),y(t))\|^2 + \frac{4L_{xy}^2\gamma\alpha^2}{\mu_{y}\beta}V_{2}(x,y)\notag\\
&\quad -\frac{\alpha}{2}\|\nabla_{x}f(x(t),y(t))\|^2  + \frac{\gamma\alpha^2}{8\beta}\|\nabla_{x} f(x(t),y(t))\|^2\notag\\
&\leq -\frac{\alpha}{2}\|\nabla_{x}f(x(t),y(t))\|^2 - \frac{\gamma\alpha}{2}\|\nabla_{y}f(x(t),y(t))\|^2\notag\\
&\quad - \frac{\gamma\alpha}{4}\|\nabla_{y}f(x(t),y(t))\|^2 + \beta\|\nabla_{y}f(x(t),y(t))\|^2\notag\\ 
&\quad - \frac{\mu_{y}\gamma\alpha}{2}\Big(1 - \frac{4L_{xy}^2\alpha}{\mu_{y}^2\beta}\Big)V_{2}(x,y) -\frac{\alpha}{2}\Big(1 - \frac{\gamma\alpha}{4\beta}\Big)\|\nabla_{x} f(x(t),y(t))\|^2,\label{one_sided_PL:thm:Eq1} 
\end{align}
where in the last inequality we use \eqref{assump:one-sided-PL:ineq} to have
\begin{align*}
\|\nabla_{y}f(x,y)\|^2 \geq 2\mu_{y}[\max_{y}f(x,y) - f(x,y)] = 2\mu_{y}V_{2}(x,y).     
\end{align*}
Using \eqref{one_sided_PL:thm:stepsize} and the preceding relation we have 
\begin{align*}
&\frac{-\gamma\alpha}{4} + \beta = -\frac{1}{\mu_{y}^2} + \frac{1}{\mu_{y}^2} = 0,\quad 1 - \frac{4L_{xy}^2\alpha}{\mu_{y}^2\beta} =  \frac{1}{2}\quad \text{and}\quad 1 - \frac{\gamma\alpha}{4\beta} = 0,
\end{align*}
which when substituting into \eqref{one_sided_PL:thm:Eq1} gives
\begin{align*}
\dot{V}(x(t),y(t)) &\leq -\frac{\alpha}{2}\|\nabla_{x}f(x(t),y(t))\|^2 - \frac{\gamma\alpha}{2}\|\nabla_{y}f(x(t),y(t))\|^2.
\end{align*}
Taking the integral on both sides over $t\in[0,T]$ for some $T\geq 0$ and rearranging we obtain
\begin{align*}
\frac{\alpha}{2}\int_{t=0}^{T}\|\nabla_{x}f(x(t),y(t))\|^2dt + \frac{\gamma\alpha}{2}\int_{t=0}^{T}\|\nabla_{y}f(x(t),y(t))\|^2dt &\leq V(x(0),y(0)),
\end{align*}
which since $\gamma \geq 1$ and by using \eqref{one_sided_PL:thm:stepsize} gives
\begin{align*}
\min_{0\leq t\leq T}\left\|\begin{array}{cc}
 \nabla_{x}f(x(t),y(t)) \\
 \nabla_{y}f(x(t),y(t)) 
\end{array}\right\| &\leq \sqrt{\frac{2V(x(0),y(0))}{\alpha T}}\notag\\
&\leq \frac{4L_{xy}\sqrt{V_{1}(x(0),y(0)) +4V_{2}(x(0),y(0))}}{\sqrt{T}},
\end{align*}
which concludes our proof. 
\end{proof}


\subsection{Nonconvex--Strongly Concave Conditions}\label{subsec:StronglyConcave}
In this subsection, we study the rate of GDAD when the function $f(x,y)$ is nonconvex given any $y$ and strongly concave given any $x$. In particular, we consider the following assumption. 
\begin{assump}\label{assump:stronglyconcave}
The objective function $f(\cdot,y)$ is nonconvex for any given $y$ and $f(x,\cdot)$ is strongly concave with constant $\mu_{y} > 0$ for any given $x$. The latter is equivalent to 
\begin{align}
f(x,y_{1}) - f(x,y_{2}) - \langle\nabla f(x,y_{2}), y_{1} - y_{2}\rangle \leq -\frac{\mu_{y}}{2}\|y_{1}-y_{2}\|^2,\quad \forall y_{1},y_{2}\in\Rset^{n}.\label{assump:stronglyconcave:ineq}   
\end{align}
\end{assump}
For our analysis of in this section, we introduce the following two Lyapunov functions
\begin{align}
V_{1}(x,y) &= f(x,y) - \min_{(x,y)}f(x,y)\label{stronglyconcave:Lyapunov:V1}\\
V_{2}(x,y) &= \frac{1}{2}\|\dot{y}\|^2 = \frac{1}{2}\|\beta\nabla_{y}f(x,y)\|^2.\label{stronglyconcave:Lyapunov:V2}
\end{align}
The time derivatives of $V_{1}$ and $V_{2}$ over the trajectories $\xdot$ and $\ydot$ are given in the following lemma, whose proof is presented in Section \ref{analysis:proof:stronglyconcave:lem:V12_dot}. 
\begin{lem}\label{stronglyconcave:lem:V12_dot}
Suppose that Assumptions \ref{assump:Lipschitz_cont} and \ref{assump:stronglyconcave} hold. Then we have
\begin{align}
\dot{V}_{1}(x(t),y(t)) &\leq -\frac{1}{\alpha}\|\xdot(t)\|^2 + \frac{1}{\beta}\|\ydot(t)\|^2.   \label{stronglyconcave:lem:V12_dot:ineq1}\\
\dot{V}_{2}(x(t),y(t)) &\leq  L_{xy}\beta\|\ydot(t)\|\|\xdot(t)\| - \mu_{y}\beta\|\ydot(t)\|^2  \label{stronglyconcave:lem:V12_dot:ineq2}.
\end{align}
\end{lem}
We next derive the convergence rate of GDAD under Assumption \ref{assump:stronglyconcave} in the following theorem, where we show that GDAD converges sublinear to a stationary point of $f$. 


\begin{thm}\label{stronglyconcave:thm}
Suppose that Assumptions \ref{assump:Lipschitz_cont} and \ref{assump:stronglyconcave} hold. Let $\gamma,\alpha,\beta$ be chosen as
\begin{align}
\gamma = \mu_{y}L_{xy}^2,\quad \alpha = \frac{1}{L_{xy}^2},\quad \beta = \frac{4}{\mu_{y}^2}.    \label{stronglyconcave:thm:stepsize}
\end{align}
Then we have for all $T\geq 0$
\begin{align}
\min_{0\leq t\leq T}\left\|\begin{array}{cc}
 \nabla_{x}f(x(t),y(t)) \\
 \nabla_{y}f(x(t),y(t)) 
\end{array}\right\| &\leq  \frac{L_{xy} \sqrt{2V_{1}(x(0),y(0))}}{\sqrt{T}} +  \frac{2L_{xy} \|\nabla_{y}f(x(0),y(0))\|}{\sqrt{\mu_{y} T}}\cdot    \label{stronglyconcave:thm:rate}
\end{align}
\end{thm}

\begin{proof}
By using \eqref{stronglyconcave:lem:V12_dot:ineq1} and \eqref{stronglyconcave:lem:V12_dot:ineq2} we consider
\begin{align}
\dot{V}(x(t),y(t)) &= \dot{V}_{1}(x(t),y(t)) + \frac{\gamma\alpha}{\beta}\dot{V}_{2}(x(t),y(t))\notag\\
&\leq -\frac{1}{\alpha}\|\xdot(t)\|^2 + \frac{1}{\beta}\|\ydot(t)\|^2 + L_{xy}\gamma \alpha\|\ydot(t)\|\|\xdot(t)\| - \mu_{y}\gamma\alpha\|\ydot(t)\|^2\notag\\
&= -\frac{1}{2\alpha}\|\xdot(t)\|^2 - \frac{\mu_{y}\gamma\alpha}{4}\|\ydot(t)\|^2 - (\frac{\mu_{y}\gamma\alpha}{4} - \frac{1}{\beta})\|\ydot(t)\|^2\notag\\
&\quad -\frac{1}{2\alpha}\|\xdot(t)\|^2 + L_{xy}\gamma \alpha\|\ydot(t)\|\|\xdot(t)\| -   \frac{\mu_{y}\gamma\alpha}{2}\|\ydot(t)\|^2\notag\\
&= -\frac{1}{2\alpha}\|\xdot(t)\|^2 - \frac{\mu_{y}\gamma\alpha}{4}\|\ydot(t)\|^2 \notag\\
&\quad -\frac{1}{2\alpha}\|\xdot(t)\|^2 + L_{xy}\gamma \alpha\|\ydot(t)\|\|\xdot(t)\| -   \frac{\mu_{y}\gamma\alpha}{2}\|\ydot(t)\|^2,\label{stronglyconcave:thm:Eq1}
\end{align}
where in the last equality we use \eqref{stronglyconcave:thm:stepsize} to have 
\begin{align*}
&\frac{\mu_{y}\gamma\alpha}{4} - \frac{1}{\beta} = \frac{\mu_{y}^2}{4} - \frac{\mu_{y}^2}{4} = 0.
\end{align*}
Using \eqref{stronglyconcave:thm:stepsize} one more time we obtain
\begin{align*}
&-\frac{1}{2\alpha}\|\xdot(t)\|^2 + L_{xy}\gamma \alpha\|\ydot(t)\|\|\xdot(t)\| -   \frac{\mu_{y}\gamma\alpha}{2}\|\ydot(t)\|^2\\
&= -\frac{1}{2\alpha}\|\xdot(t)\|^2 + \mu_{y}L_{xy}\|\ydot(t)\|\|\xdot(t)\| -   \frac{\mu_{y}^2L_{xy}^2\alpha}{2}\|\ydot(t)\|^2\leq 0,  
\end{align*}
which when using into \eqref{stronglyconcave:thm:Eq1} we obtain
\begin{align*}
\dot{V}(x(t),y(t)) &\leq    -\frac{1}{2\alpha}\|\xdot(t)\|^2 - \frac{\mu_{y}\gamma\alpha}{4}\|\ydot(t)\|^2\\
&= -\frac{\alpha}{2}\|\nabla_{x}f(x(t),y(t))\|^2 -\frac{\mu_{y}\gamma\alpha\beta^2}{4}\|\nabla_{y}f(x(t),y(t))\|^2\\
&= -\frac{\alpha}{2}\|\nabla_{x}f(x(t),y(t))\|^2 -\frac{4L_{xy}^2\alpha}{\mu_{y}^2}\|\nabla_{y}f(x(t),y(t))\|^2\\
&\leq \frac{-\alpha}{2}\big(\|\nabla_{x}f(x(t),y(t))\|^2 +\|\nabla_{y}f(x(t),y(t))\|^2\big),
\end{align*}
which when taking the integral on both sides over $t$ from $0$ to $T$ and rearrange we obtain 
\begin{align*}
\frac{\alpha}{2}\int_{t=0}^{T}\big(\|\nabla_{x}f(x(t),y(t))\|^2 +\|\nabla_{y}f(x(t),y(t))\|^2\big)dt \leq V(x(0),y(0)).    
\end{align*}
Thus, the preceding relation gives \eqref{stronglyconcave:thm:rate}, i.e., for all $T>0$
\begin{align*}
\min_{0\leq t\leq T}\left\|\begin{array}{cc}
 \nabla_{x}f(x(t),y(t)) \\
 \nabla_{y}f(x(t),y(t)) 
\end{array}\right\| &\leq \frac{\sqrt{2V(x(0),y(0))}}{\sqrt{\alpha T}} \leq  \frac{\sqrt{2V_{1}(x(0),y(0))}}{\sqrt{\alpha T}} + \frac{\sqrt{2V_{2}(x(0),y(0))}}{\sqrt{\alpha T}}\notag\\
&= \sqrt{\frac{2V_{1}(x(0),y(0))}{\alpha T}} +  \sqrt{\frac{\gamma\alpha\beta \|\nabla_{y}f(x(0),y(0))\|^2}{\alpha T}}\notag\\
&= \frac{L_{xy} \sqrt{2V_{1}(x(0),y(0))}}{\sqrt{T}} +  \frac{2L_{xy} \|\nabla_{y}f(x(0),y(0))\|}{\sqrt{\mu_{y} T}}\cdot
\end{align*}
\end{proof}



\subsection{Strongly Convex--Nonconcave Conditions}\label{subsec:stronglyconvex}
As mentioned, the single-loop GDA method is applicable to the convex-nonconcave min-max problem, while the nested-loop GDA method is not. In this section, we complete our analysis by studying the rate of GDAD when the function $f(x,y)$ is strongly convex given any $y$ and  nonconcave given any $x$. In particular, we consider the following assumption. 
\begin{assump}\label{assump:stronglyconvex}
The objective function $f(x,\cdot)$ is nonconcave for any given $x$ and $f(\cdot,y)$ is strongly convex with constant $\mu_{x} > 0$ for any given $y$. The latter is equivalent to
\begin{align}
f(x_{1},y) - f(x_{2},y) - \langle\nabla f(x_{2},y), x_{1} - x_{2}\rangle \geq \frac{\mu_{x}}{2}\|x_{1}-x_{2}\|^2,\quad \forall x_{1},x_{2}\in\Rset^{m}.\label{assump:stronglyconvex:ineq}   
\end{align}
\end{assump}
For our analysis of in this section, we introduce the following two Lyapunov functions
\begin{align}
V_{1}(x,y) &= \frac{1}{2}\|\dot{x}\|^2 = \frac{1}{2}\|\alpha\nabla_{x}f(x,y)\|^2\label{stronglyconvex:Lyapunov:V1}\\
V_{2}(x,y) &= \max_{x,y}f(x,y) - f(x,y).\label{stronglyconvex:Lyapunov:V2}
\end{align}
The time derivatives of $V_{1}$ and $V_{2}$ over the trajectories $\xdot$ and $\ydot$ are given in the following lemma, whose proof is presented in Section \ref{analysis:proof:stronglyconvex:lem:V12_dot}. 
\begin{lem}\label{stronglyconvex:lem:V12_dot}
Suppose that Assumptions \ref{assump:Lipschitz_cont} and \ref{assump:stronglyconvex} hold. Then we have
\begin{align}
\dot{V}_{1}(x(t),y(t)) &\leq - \mu_{x}\alpha\|\xdot(t)\|^2 + L_{xy}\alpha\|\ydot(t)\|\|\xdot(t)\|.   \label{stronglyconvex:lem:V12_dot:ineq1}\\
\dot{V}_{2}(x(t),y(t)) &\leq  \frac{1}{\alpha}\|\xdot(t)\|^2 - \frac{1}{\beta}\|\ydot(t)\|^2 \label{stronglyconvex:lem:V12_dot:ineq2}.
\end{align}
\end{lem}
Since in the strongly convex-nonconcave setting the minimization problem is easier to solve than the maximization problem, we consider the following Lyapunov function
\begin{align}
V(x,y) = V_{2}(x,y) + \frac{\gamma\beta}{\alpha} V_{1}(x,y),\label{def:Lyapunov_V_sc}  
\end{align}
where $\beta \ll \alpha$. In this case, $\xdot$ is updated at a faster time scale than $\ydot$. Using this Lyapunov function, we now derive the convergence rate of GDAD under Assumption \ref{assump:stronglyconvex} in the following theorem, which basically is similar to the one in Theorem \ref{stronglyconcave:thm}. 

\begin{thm}\label{stronglyconvex:thm}
Suppose that Assumptions \ref{assump:Lipschitz_cont} and \ref{assump:stronglyconvex} hold. Let $\gamma,\alpha,\beta$ be chosen as
\begin{align}
\gamma = \mu_{x}L_{xy}^2,\quad \alpha = \frac{4}{\mu_{x}^2},\quad \beta =  \frac{1}{L_{xy}^2}.    \label{stronglyconvex:thm:stepsize}
\end{align}
Then we have for all $T\geq 0$
\begin{align}
\min_{0\leq t\leq T}\left\|\begin{array}{cc}
 \nabla_{x}f(x(t),y(t)) \\
 \nabla_{y}f(x(t),y(t)) 
\end{array}\right\| &\leq  \frac{L_{xy} \sqrt{2V_{1}(x(0),y(0))}}{\sqrt{T}} +  \frac{2L_{xy} \|\nabla_{x}f(x(0),y(0))\|}{\sqrt{\mu_{x} T}}\cdot    \label{stronglyconvex:thm:rate}
\end{align}
\end{thm}

\begin{proof}
By using \eqref{stronglyconvex:lem:V12_dot:ineq1} and \eqref{stronglyconvex:lem:V12_dot:ineq2} we consider
\begin{align}
\dot{V}(x(t),y(t)) &= \dot{V}_{2}(x(t),y(t)) + \frac{\gamma\beta}{\alpha}\dot{V}_{1}(x(t),y(t))\notag\\
&\leq \frac{1}{\alpha}\|\xdot(t)\|^2 - \frac{1}{\beta}\|\ydot(t)\|^2 + L_{xy}\gamma \beta\|\ydot(t)\|\|\xdot(t)\| - \mu_{x}\gamma\beta\|\xdot(t)\|^2\notag\\
&= -\frac{1}{2\beta}\|\ydot(t)\|^2 - \frac{\mu_{x}\gamma\beta}{4}\|\xdot(t)\|^2 - (\frac{\mu_{x}\gamma\beta}{4} - \frac{1}{\alpha})\|\xdot(t)\|^2\notag\\
&\quad -\frac{1}{2\beta}\|\ydot(t)\|^2 + L_{xy}\gamma \beta\|\ydot(t)\|\|\xdot(t)\| -   \frac{\mu_{x}\gamma\beta}{2}\|\xdot(t)\|^2\notag\\
&= -\frac{1}{2\beta}\|\ydot(t)\|^2 - \frac{\mu_{x}\gamma\beta}{4}\|\xdot(t)\|^2 \notag\\
&\quad -\frac{1}{2\beta}\|\ydot(t)\|^2 + L_{xy}\gamma \beta\|\ydot(t)\|\|\xdot(t)\| -   \frac{\mu_{x}\gamma\beta}{2}\|\xdot(t)\|^2,\label{stronglyconvex:thm:Eq1}
\end{align}
where in the last equality we use \eqref{stronglyconvex:thm:stepsize} to have 
\begin{align*}
&\frac{\mu_{x}\gamma\beta}{4} - \frac{1}{\alpha} = \frac{\mu_{x}^2}{4} - \frac{\mu_{x}^2}{4} = 0.
\end{align*}
Using \eqref{stronglyconvex:thm:stepsize} one more time we obtain
\begin{align*}
&-\frac{1}{2\beta}\|\ydot(t)\|^2 + L_{xy}\gamma \beta\|\ydot(t)\|\|\xdot(t)\| -   \frac{\mu_{x}\gamma\beta}{2}\|\xdot(t)\|^2\\
&= -\frac{1}{2\beta}\|\ydot(t)\|^2 + \mu_{x}L_{xy}\|\ydot(t)\|\|\xdot(t)\| -   \frac{\mu_{x}^2L_{xy}^2\beta}{2}\|\xdot(t)\|^2\leq 0,  
\end{align*}
which when using into \eqref{stronglyconvex:thm:Eq1} we obtain
\begin{align*}
\dot{V}(x(t),y(t)) &\leq    -\frac{1}{2\beta}\|\ydot(t)\|^2 - \frac{\mu_{x}\gamma\beta}{4}\|\xdot(t)\|^2\\
&= -\frac{\beta}{2}\|\nabla_{y}f(x(t),y(t))\|^2 -\frac{\mu_{x}\gamma\beta\alpha^2}{4}\|\nabla_{x}f(x(t),y(t))\|^2\\
&= -\frac{\beta}{2}\|\nabla_{y}f(x(t),y(t))\|^2 -\frac{4L_{xy}^2\beta}{\mu_{x}^2}\|\nabla_{x}f(x(t),y(t))\|^2\\
&\leq \frac{-\beta}{2}\big(\|\nabla_{x}f(x(t),y(t))\|^2 +\|\nabla_{y}f(x(t),y(t))\|^2\big),
\end{align*}
which when taking the integral on both sides over $t$ from $0$ to $T$ and rearrange we obtain 
\begin{align*}
\frac{\beta}{2}\int_{t=0}^{T}\big(\|\nabla_{x}f(x(t),y(t))\|^2 +\|\nabla_{y}f(x(t),y(t))\|^2\big)dt \leq V(x(0),y(0)).    
\end{align*}
Thus, the preceding relation gives \eqref{stronglyconvex:thm:rate}, i.e., for all $T>0$
\begin{align*}
\min_{0\leq t\leq T}\left\|\begin{array}{cc}
 \nabla_{x}f(x(t),y(t)) \\
 \nabla_{y}f(x(t),y(t)) 
\end{array}\right\| &\leq \frac{\sqrt{2V(x(0),y(0))}}{\sqrt{\beta T}} \leq  \frac{\sqrt{2V_{1}(x(0),y(0))}}{\sqrt{\beta T}} + \frac{\sqrt{2V_{2}(x(0),y(0))}}{\sqrt{\beta T}}\notag\\
&= \sqrt{\frac{2V_{1}(x(0),y(0))}{\beta T}} +  \sqrt{\frac{\gamma\beta\alpha \|\nabla_{x}f(x(0),y(0))\|^2}{\beta T}}\notag\\
&= \frac{L_{xy} \sqrt{2V_{1}(x(0),y(0))}}{\sqrt{T}} +  \frac{2L_{xy} \|\nabla_{x}f(x(0),y(0))\|}{\sqrt{\mu_{x} T}}\cdot
\end{align*}
\end{proof}


\section{Proofs of Technical Lemmas}
In this section, we present the analysis of all technical lemmas in the previous sections.

\subsection{Proof of Lemma \ref{two-sided-PL:lem:V12_dot}}\label{analysis:proof:two-sided-PL:lem:V12_dot}
\begin{proof}
For convenience, we denote by $y^{\star}(x) = \Pcal_{\Ycal^{\star}(x)}[y]$, where recall that $\Ycal^{\star}(x)$ is the solution set of $\max_{y}f(x,y)$ for a given $x$. We first show \eqref{two-sided-PLlem:V12_dot:ineq1}. The time derivative of $V_{1}$ defined in \eqref{two-sided-PLLyapunov_V1} over the trajectory $\xdot$ in \eqref{alg:GDAD} is given as
\begin{align}
\dot{V}_{1}(x) &= \frac{d}{dt}V_{1}(x) = \nabla_{x} f(x,y^{\star}(x))\xdot = -\alpha\langle\nabla_{x} f(x,y^{\star}(x)),\nabla_{x} f(x,y) \rangle\notag\\
&= -\alpha\|\nabla_{x} f(x,y^{\star}(x))\|^2 -\alpha\langle\nabla_{x} f(x,y^{\star}(x)),\nabla_{x} f(x,y) -  \nabla_{x} f(x,y^{\star}(x))\rangle\notag\\
&\leq -\frac{\alpha}{2}\|\nabla_{x} f(x,y^{\star}(x))\|^2 + \frac{\alpha}{2}\|\nabla_{x} f(x,y) -  \nabla_{x} f(x,y^{\star}(x))\|^2\label{two-sided-PLlem:V12_dot:Eq1}\\
&\leq -\frac{\alpha}{2}\|\nabla_{x} f(x,y^{\star}(x))\|^2 + \frac{L_{xy}^2\alpha}{2}\|y-y^{\star}(x)\|^2\notag\\
&\leq -\frac{\alpha}{2}\|\nabla_{x} f(x,y^{\star}(x))\|^2 +  \frac{L_{xy}^2\alpha}{\mu_{y}}(\max_{y}f(x,y) - f(x,y))\notag,
\end{align}
where the first and second inequalities are due to the Cauchy-Schwarz inequality and Assumption \ref{assump:Lipschitz_cont}, respectively. In addition, the last inequality is due to \eqref{eq:QG}.

Next we show \eqref{two-sided-PLlem:V12_dot:ineq2}. Indeed, using \eqref{two-sided-PLlem:V12_dot:Eq1} we have
\begin{align*}
\dot{V}_{2}(x,y) &= \nabla_{x}f(x,y^{\star}(x))\xdot -\nabla_{x}f(x,y)\xdot - \nabla_{y}f(x,y)\ydot\\  
&\leq  -\frac{\alpha}{2}\|\nabla_{x} f(x,y^{\star}(x))\|^2 + \frac{\alpha}{2}\|\nabla_{x} f(x,y) -  \nabla_{x} f(x,y^{\star}(x))\|^2  - \beta\|\nabla_{y}f(x,y)\|^2  \notag\\ 
&\quad+ \alpha\|\nabla_{x}f(x,y)\|^2\notag\\
&\leq -\frac{\alpha}{2}\|\nabla_{x} f(x,y^{\star}(x))\|^2  - \beta\|\nabla_{y}f(x,y)\|^2 + \frac{\alpha}{2}\|\nabla_{x} f(x,y) -  \nabla_{x} f(x,y^{\star}(x))\|^2 \notag\\ 
&\quad + 2\alpha\|\nabla_{x}f(x,y) - \nabla_{x} f(x,y^{\star}(x))\|^2 + 2\alpha \|\nabla_{x} f(x,y^{\star}(x))\|^2\notag\\
&=  - \beta\|\nabla_{y}f(x,y)\|^2 + \frac{3\alpha}{2} \|\nabla_{x} f(x,y^{\star}(x))\|^2+ \frac{5\alpha}{2}\|\nabla_{x} f(x,y) -  \nabla_{x} f(x,y^{\star}(x))\|^2\notag\\
&\leq - \beta\|\nabla_{y}f(x,y)\|^2 + \frac{3\alpha}{2} \|\nabla_{x} f(x,y^{\star}(x))\|^2+ \frac{5L_{xy}^2\alpha}{\mu_{y}}(\max_{y}f(x,y) - f(x,y)),
\end{align*}
where the last inequality we use Assumption \ref{assump:Lipschitz_cont} and \eqref{eq:QG}, similar to the last inequality in $\dot{V}_{1}$ above.
\end{proof}

\subsection{Proof of Lemma \ref{One_sided_PL:lem:V12_dot}}\label{analysis:proof:one-sided-PL:lem:V12_dot}
\begin{proof}
For convenience, we denote by $y^{\star}(x) = \Pcal_{\Ycal^{\star}(x)}[y]$, where recall that $\Ycal^{\star}(x)$ is the solution set of $\max_{y}f(x,y)$ for a given $x$. We first show \eqref{One_sided_PL:lem:V12_dot:ineq1}. The time derivative of $V_{1}$ defined in \eqref{One_sided_PL:Lyapunov_V1} over the trajectory $\xdot$ in \eqref{alg:GDAD} is given as
\begin{align*}
\dot{V}_{1}(x(t),y(t)) &= \frac{d}{dt}V_{1}(x(t),y(t)) = \nabla_{x}f(x(t),y(t))\xdot + \nabla_{y}f(x(t),y(t))\ydot\\ 
&= -\alpha\|\nabla_{x}f(x(t),y(t))\|^2 + \beta\|\nabla_{y}f(x(t),y(t))\|^2.
\end{align*}
Next we show \eqref{One_sided_PL:lem:V12_dot:ineq2}. Indeed, using \eqref{alg:GDAD} we have
\begin{align*}
\dot{V}_{2}(x(t),y(t)) &= \nabla_{x}f(x(t),y^{\star}(x(t)))\xdot -\nabla_{x}f(x(t),y(t))\xdot - \nabla_{y}f(x(t),y(t))\ydot\\  
&=  -\alpha\langle\nabla_{x} f(x(t),y^{\star}(x(t))),\nabla_{x} f(x(t),y(t)) \rangle   - \beta\|\nabla_{y}f(x(t),y(t))\|^2   \notag\\ 
&\quad+ \alpha\|\nabla_{x}f(x(t),y(t))\|^2\notag\\
&= -\alpha\|\nabla_{x} f(x(t),y(t)) \|^2   - \beta\|\nabla_{y}f(x(t),y(t))\|^2+ \alpha\|\nabla_{x}f(x(t),y(t))\|^2\notag\\ 
&\quad  -\alpha\langle\nabla_{x} f(x(t),y^{\star}(x(t)))-\nabla_{x} f(x(t),y(t)) ,\nabla_{x} f(x(t),y(t)) \rangle \notag\\
&\leq - \beta\|\nabla_{y}f(x(t),y(t))\|^2 + \frac{\alpha}{8}\|\nabla_{x} f(x(t),y(t))\|^2\notag\\ 
&\quad + 2\alpha\|\nabla_{x} f(x(t),y^{\star}(x(t)))-\nabla_{x} f(x(t),y(t)) \|^2\notag\\
&\leq - \beta\|\nabla_{y}f(x(t),y(t))\|^2 + \frac{\alpha}{8}\|\nabla_{x} f(x(t),y(t))\|^2 + 2L_{xy}^2\alpha\|y^{\star}(x(t))-y(t) \|^2\notag\\
&\leq  - \beta\|\nabla_{y}f(x(t),y(t))\|^2 + \frac{\alpha}{8}\|\nabla_{x} f(x(t),y(t))\|^2 + \frac{4L_{xy}^2\alpha}{\mu_{y}}V_{2}(x(t),y(t)),
\end{align*}
where the second inequality is due to the Cauchy-Schwarz inequality: $2xy\leq (1/\eta) x^2 + \eta y^2$ for any $\eta > 0$. In addition, the last two inequalities are due to Eqs.\ \ref{assump:Lipschitz_cont:ineq} and \eqref{eq:QG}, respectively.
\end{proof}

\subsection{Proof of Lemma \ref{stronglyconcave:lem:V12_dot}}\label{analysis:proof:stronglyconcave:lem:V12_dot}
\begin{proof}
We first show \eqref{stronglyconcave:lem:V12_dot:ineq1}. Indeed, using \eqref{alg:GDAD}
\begin{align*}
\dot{V}_{1}(x(t),y(t)) &= \frac{d}{dt}V_{1}(x(t),y(t)) = \nabla_{x}f(x(t),y(t))\xdot + \nabla_{y}f(x(t),y(t))\ydot\\ 
&= -\frac{1}{\alpha}\|\xdot(t)\|^2 + \frac{1}{\beta}\|\ydot(t)\|^2.        
\end{align*}
Next, we show \eqref{stronglyconcave:lem:V12_dot:ineq2}. Using \eqref{alg:GDAD} and \eqref{stronglyconcave:Lyapunov:V2} we consider
\begin{align*}
\dot{V}_{2}(x(t),y(t)) &= \beta\langle \ydot(t),\frac{d}{dt}\nabla_{y}f(x,y) \rangle\\
&= \beta\langle \ydot(t), \nabla_{yx}f(x(t),y(t))\xdot(t) + \nabla_{yy}f(x(t),y(t))\ydot(t)\rangle\\
&= \beta\langle \ydot(t), \nabla_{yx}f(x(t),y(t))\xdot(t)\rangle + \beta\langle \ydot(t), \nabla_{yy}f(x(t),y(t))\ydot(t)\rangle\\
&\leq L_{xy}\beta\|\ydot(t)\|\|\xdot(t)\| - \mu_{y}\beta\|\ydot(t)\|^2,   
\end{align*}
where in the last inequality we use Assumptions \ref{assump:Lipschitz_cont} and \ref{assump:stronglyconcave} to have
\begin{align*}
\|\nabla_{yx}f(x(t),y(t))\| \leq L_{xy}\quad \text{and}\quad \nabla_{yy}f(x(t),y(t)) \leq -\mu_{y}\Ibf,\;\text{where } \Ibf \text{ is the identity matrix}.      
\end{align*}
\end{proof}

\subsection{Proof of Lemma \ref{stronglyconvex:lem:V12_dot}}\label{analysis:proof:stronglyconvex:lem:V12_dot}
\begin{proof}
We first show \eqref{stronglyconvex:lem:V12_dot:ineq1}. Using \eqref{alg:GDAD} and \eqref{stronglyconvex:Lyapunov:V1} we consider
\begin{align*}
\dot{V}_{1}(x(t),y(t)) &= -\alpha\langle \xdot(t),\frac{d}{dt}\nabla_{x}f(x(t),y(t)) \rangle\\
&= -\alpha\langle \xdot(t), \nabla_{xx}f(x(t),y(t))\xdot(t) + \nabla_{xy}f(x(t),y(t))\ydot(t)\rangle\\
&= -\alpha\langle \xdot(t), \nabla_{xx}f(x(t),y(t))\xdot(t)\rangle - \alpha\langle \xdot(t), \nabla_{xy}f(x(t),y(t))\ydot(t)\rangle\\
&\leq - \mu_{x}\alpha\|\xdot(t)\|^2 + L_{xy}\alpha\|\ydot(t)\|\|\xdot(t)\| ,   
\end{align*}
where in the last inequality we use Assumptions \ref{assump:Lipschitz_cont} and \ref{assump:stronglyconvex} to have
\begin{align*}
\|\nabla_{xy}f(x(t),y(t))\| \leq L_{xy}\quad \text{and}\quad \nabla_{xx}f(x(t),y(t)) \geq \mu_{x}\Ibf,\;\text{where } \Ibf \text{ is the identity matrix}.      
\end{align*} 
Next, we show  \eqref{stronglyconvex:lem:V12_dot:ineq2}. Indeed, using \eqref{alg:GDAD}
\begin{align*}
\dot{V}_{2}(x(t),y(t)) &= \frac{d}{dt}V_{2}(x(t),y(t)) = -\nabla_{x}f(x(t),y(t))\xdot - \nabla_{y}f(x(t),y(t))\ydot\\ 
&= \frac{1}{\alpha}\|\xdot(t)\|^2 - \frac{1}{\beta}\|\ydot(t)\|^2.        
\end{align*}
\end{proof}

\section{Concluding Remarks}
In this paper, we consider two-time-scale gradient descent-ascent dynamics for solving nonconvex min-max optimization problems. Our main focus is to derive the convergence rates of this method for different settings of the underlying objective functions. Our techniques are mainly motivated by the classic singular perturbation, where we show that our analysis improves the existing results under the same conditions. A natural extension from this work  is to provide a better analysis for the discrete-time variant of GDAD. Another interesting future direction is to consider the stochastic setting and its accelerated counterpart.   

\bibliographystyle{IEEEtran}
\bibliography{refs}

\end{document}